\documentclass[12pt]{article}
\usepackage[T1]{fontenc}
\usepackage{a4wide}
\usepackage{amsthm}
\usepackage{amsfonts}
\usepackage{amssymb}
\usepackage{amsmath}
\usepackage{cite}
\usepackage{epsfig}
\usepackage{tabularx}
\usepackage{mathtools}
\usepackage[colorlinks=false]{hyperref}
\usepackage{nameref,cleveref}
\usepackage{float}
\usepackage{enumerate}

\newtheorem{theorem}{Theorem}

\newtheorem{lemma}[theorem]{Lemma}
\newtheorem{proposition}[theorem]{Proposition}
\newtheorem{problem}{Problem}
\newtheorem{conjecture}[problem]{Conjecture}
\newcommand\MN{{\cal N}}
\newcommand\MM{{\cal M}}
\newcommand\FF{{\mathbb F}}

\newcommand\OO{\mathcal{O}}
\DeclareMathOperator{\rank}{rank}
\DeclareMathOperator{\cdd}{cdd}
\DeclareMathOperator{\bd}{bd}
\DeclareMathOperator{\cd}{cd}
\DeclareMathOperator{\dd}{dd}

\DeclareTextCompositeCommand{\v}{OT1}{l}{l\nobreak\hspace{-.1em}'}

\DeclarePairedDelimiter\set{\{}{\}}

\DeclarePairedDelimiter\paren{(}{)}

\makeatletter
\def\@fnsymbol#1{\ensuremath{\ifcase#1\or \dagger\or \ddagger\or%
\mathsection\or \mathparagraph\or \|\or **\or \dagger\dagger%
\or \ddagger\ddagger \else\@ctrerr\fi}}%
\makeatother

\begin{document}

\title{Branch-depth is minor closure of contraction-deletion-depth\thanks{The second and third authors have been supported by the project 24-11098S of the Czech Science Foundation.}}

\author{Marcin Bria\'nski\thanks{Theoretical Computer Science Department, Faculty of Mathematics and Computer Science, Jagiellonian University, Krak\'ow, Poland. E-mail: \texttt{marcin.brianski@doctoral.uj.edu.pl}.}\and
        Daniel Kr{\'a}\v{l}\thanks{Faculty of Informatics, Masaryk University, Botanick\'a 68A, 602 00 Brno, Czech Republic. E-mails: {\tt dkral@fi.muni.cz} and {\tt kristyna.pekarkova@mail.muni.cz}.}\and
\newcounter{lth}
\setcounter{lth}{3}
        Krist\'yna Pek\'arkov\'a$^\fnsymbol{lth}$}
\date{} 
\maketitle

\begin{abstract}
The notion of branch-depth for matroids was introduced by DeVos, Kwon and Oum 
as the matroid analogue of the tree-depth of graphs.
The contraction-deletion-depth, another tree-depth like parameter of matroids,
is the number of recursive steps needed to decompose a matroid by contractions and deletions to single elements.
Any matroid with contraction-deletion-depth at most $d$ has branch-depth at most $d$.
However, the two notions are not functionally equivalent
as contraction-deletion-depth of matroids with branch-depth two can be arbitrarily large.

We show that the two notions are functionally equivalent for representable matroids when minor closures are considered.
Namely, an $\FF$-representable matroid has small branch-depth if and only if
it is a minor of an $\FF$-representable matroid with small contraction-deletion-depth.
This implies that any class of $\FF$-representable matroids has bounded branch-depth if and only if 
it is a subclass of the minor closure of a class of $\FF$-representable matroids with bounded contraction-deletion-depth.
\end{abstract}

\section{Introduction}
\label{sec:intro}

Graph width parameters play an important role both in structural and algorithmic graph theory.
Indeed, graph tree-width, likely the best-known graph width parameter,
is one of the core notions in the Graph Minor Project of Robertson and Seymour.
The celebrated Courcelle's Theorem~\cite{Cou90} asserts that
every monadic second order property can be tested in linear times for graphs with bounded tree-width;
we refer the reader e.g. to the monograph~\cite{CygFKLMPPS15} for further algorithmic applications of tree-width and
other graph width parameters.
In this paper we study matroid branch-depth, which was introduced by DeVos, Kwon and Oum~\cite{DevKO20} as
the matroid analogue of tree-depth, a crucial structural graph measure with deep connections 
to sparsity theory of Ne{\v{s}}etril and Ossona de Mendez~\cite{NesO12},
graph product structure theorems~\cite{DvoW22}, clustered coloring~\cite{NorSW23}, and many other
areas of graph theory.
We remark that
every graph with small rank-depth is a vertex-minor of a graph with small tree-depth~\cite{HliKOO16}, and
graph classes with bounded rank-depth coincide with those with bounded shrub-depth~\cite{DevKO20},
which is a graph depth parameter inspired by the classical notion of graph clique-width.
Indeed, these two results are counterparts of more classical results concerning clique-width:
every graph with small clique-width is a pivot-minor of a graph with small tree-width~\cite{KwoO14}, and
graph classes with bounded clique-width coincide with those with bounded rank-width~\cite{OumS06} (also see~\cite{Oum05} for the relation of rank-width, pivot-minors and vertex-minors).

\emph{Branch-width} is generally considered to be the matroid analogue of graph tree-width, and
indeed Hlin\v en\'y~\cite{Hli03a,Hli03b,Hli06} proved the matroid analogue of Courcelle's Theorem
by showing that every monadic second order property can be tested in polynomial time
for matroids represented over a finite field that have bounded branch-width;
additional algorithmic results involving branch-width that parallel graph algorithms involving tree-width
can be found e.g.~\cite{GavKO12,HliO07,HliO08,JeoKO18}.
There are two matroid notions that can be considered to be analogues of graph tree-depth.
The first is the notion introduced in~\cite{KarKLM17},
which is now referred to as \emph{contraction$^*$-depth};
this parameter turned out to have an unexpected connection
to optimal tree-depth of constraint matrices in integer programs~\cite{ChaCKKP20,ChaCKKP22,BriKKPS22,BriKKPS24}.
The second is the notion of \emph{branch-depth}
introduced by DeVos, Kwon and Oum in~\cite{DevKO20} and further studied e.g. in~\cite{GolHMO21};
this notion is defined in a way completely analogous to graph rank-depth,
i.e., a matroid has small branch-depth iff
it can be represented by a shallow tree with all nodes corresponding to multicuts with small order (a formal definition
is given in Section~\ref{sec:notation}).

\begin{figure}
\begin{center}
\epsfbox{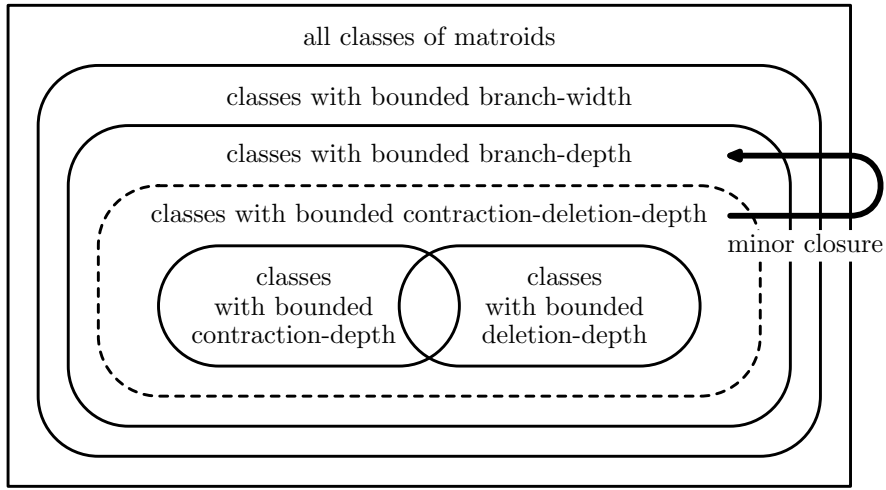}
\end{center}
\caption{The relation of classes of representable matroids with bounded depth and width parameters.
         The statement of Theorem~\ref{thm:class} is visualized by the arrow
	 relating classes with bounded contraction-deletion-depth and classes with bounded branch-depth.}
\label{fig:classes}
\end{figure}

We now briefly discuss matroid depth and matroid width parameters and their mutual relation,
which is visualized in Figure~\ref{fig:classes},
to put our results in the appropriate context.
In addition to branch-depth and branch-width,
it is possible to define the notion of \emph{contraction-depth},
which is the minimum depth of a procedure recursively splitting a given matroid by contractions to single elements,
the notion of \emph{deletion-depth},
which is the minimum depth of a procedure recursively splitting a given matroid by deletions to single elements, and
the notion of \emph{contraction-deletion-depth},
which is the minimum depth of a procedure recursively splitting a given matroid by contractions and deletions to single elements;
the formal definitions are given in Section~\ref{sec:notation}.
These notions were defined by Ding, Oporowski and Oxley~\cite{dinOO95} in 1995
although the names of contraction-depth, deletion-depth and contraction-deletion-depth were coined in~\cite{DevKO20};
concepts closely related to contraction-depth and deletion-depth were also discussed by Robertson and Seymour in~\cite{RobS85}.
We remark that, unlike branch-depth and branch-width,
none of contraction-depth, deletion-depth, and contraction-deletion-depth
is minor-monotone.

Classes of matroids with bounded contraction-depth coincide with classes of matroids with bounded contraction$^*$-depth.
In fact, the contraction$^*$-depth of a matroid $M$
is equal to the minimum contraction-depth of a matroid containing $M$ increased by one~\cite{BriKL23} (except for trivial cases
of matroids consisting of loops and bridges).
Classes of matroids with bounded contraction-depth do not need to have bounded deletion-depth and vice versa,
however,
every class of matroids with bounded contraction-depth has bounded contraction-deletion-depth, and
every class of matroids with bounded deletion-depth has bounded contraction-deletion-depth.
Furthermore, every class of matroids with bounded contraction-deletion-depth has bounded branch-depth, and
every every class of matroids with bounded branch-depth has bounded branch-width.
These relations are depicted in Figure~\ref{fig:classes} and they all are strict inclusions.

Our main result is
the following structural insight in classes of representable matroids with bounded branch-depth,
which is depicted by the arrow in Figure~\ref{fig:classes}.

\begin{theorem}
\label{thm:class}
Let $\MM$ be a class of $\FF$-representable matroids.
The class has bounded branch-depth if and only if
there exists an $\FF$-representable class $\MN$ of matroids with bounded contraction-deletion-depth such that
$\MM$ is a subclass of the minor closure of $\MN$.
\end{theorem}

\noindent In other words,
the notion of branch-depth can be thought of as the minor closure of the more restrictive parameter of contraction-deletion-depth.
In a certain sense, our result is analogous to the aforementioned result stating that
graph classes with bounded rank-depth coincide with those with bounded shrub-depth and
are vertex-minor closures of classes with bounded tree-depth.
Indeed, rank-depth captures the complexity of cuts in a manner similar to branch-depth, and
shrub-depth and tree-depth capture the complexity of decomposing a graph
while contraction-deletion-depth attempts to measure similar complexity in the matroid world.

The characterization of classes of representable matroids with bounded branch-depth given in Theorem~\ref{thm:class}
is implied by the following theorem:

\begin{theorem}
\label{thm:bound}
Every $\FF$-representable matroid $M$ is a minor of an $\FF$-representable matroid $N$ such that
\[\cdd(N)\le 2\bd(M)\cdot (4^{\bd(M)}-1)+1.\]
\end{theorem}

\noindent
Since branch-depth is minor-monotone and upper bounded by the contraction-deletion-depth,
Theorem~\ref{thm:bound} asserts that
\emph{a representable matroid has small branch-depth if and only if
it is a minor of a (representable) matroid with small contraction-deletion-depth}.
We believe that the condition on representability in Theorems~\ref{thm:class} and~\ref{thm:bound}
is an artefact of our proof methods, as further discussed in Section~\ref{sec:concl}, and
we conjecture that both these results also hold for general matroids.

\begin{figure}
\begin{center}
\epsfbox{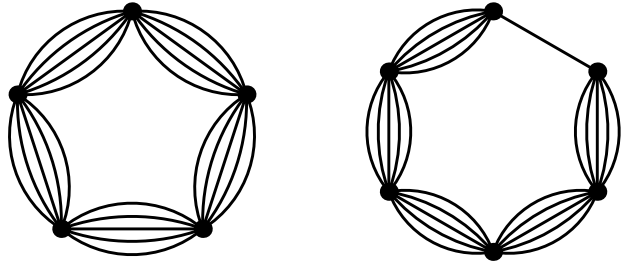}
\end{center}
\caption{A fat cycle and a graph that contains the depicted fat cycle as a minor.
         The branch-depth of the graphic matroid associated with either of the graphs is two,
	 the contraction-deletion-depth of the graphic matroid associated with the former graph is six, and
	 the contraction-deletion-depth of the graphic matroid associated with the latter graph is three.}
\label{fig:prototypical}
\end{figure}

We now try to gather some intuition
why matroids with small branch-depth are minors of matroids with small contraction-deletion-depth.
A prototypical example of a matroid with small branch-depth and large contraction-deletion-depth
is the graphic matroid of a fat cycle --- the graphic matroid associated with the graph obtained from 
a cycle of length $n$ by replacing each edge by $n$ parallel edges (see
Figure~\ref{fig:prototypical} for illustration).
The contraction-deletion depth of this graphic matroid is $n+1$ while its branch-depth is only two.
However, this graphic matroid
is a minor of the graphic matroid associated with the graph obtained
from a cycle of length $n+1$ by replacing all but one edge with $n$ parallel edges,
which has contraction-deletion-depth equal to three.
Similarly, the proof of Theorem~\ref{thm:bound} finds an embedding of a given matroid in a larger matroid that
fractures into connected components after a small number of contractions and deletions that corresponds to 
distinct subtrees in a branch-decomposition of the input matroid.
On a path towards finding such an embedding,
after reviewing the notation used in the paper in Section~\ref{sec:notation},
we prove our key structural lemma (Lemma~\ref{lm:main}) in Section~\ref{sec:lemma}.
This lemma asserts that if a collection of subspaces of a vector space has small connectivity,
then the subspaces share a low-dimensional subspace and form an almost direct sum.
In Section~\ref{sec:main},
we use the lemma to embed (as minors) matroids with small branch-depth in matroids with small contraction-deletion depth and
derive proofs of Theorems~\ref{thm:class} and~\ref{thm:bound}.
We conclude by stating open problems directly stemming from our results in Section~\ref{sec:concl}.

\section{Notation}
\label{sec:notation}

In this section, we introduce notation used throughout the paper.
The set of the first $k$ positive integers is denoted by $[k]$.
We often work with vector spaces and we briefly review notation that we use.
If $X$ is a set of vectors,
then $\overline{X}$ is the linear hull of the vectors contained in $X$.
If $X$ and $Y$ are subspaces of a vector space,
then $X+Y$ is the smallest subspace containing $X$ and $Y$, and
more generally,
if $X_1,\ldots,X_k$ are subspaces,
then $\sum\limits_{i=1}^k X_i$ is the smallest subspace containing $X_1,\ldots,X_k$.
Finally, if $X$ and $Y$ are subspaces,
the quotient space $X/Y$ is the vector space formed by classes of the equivalence relation $\sim$
such that $x\sim x'$ iff $x-x'\in Y$.
The quotient space $X/Y$ can be identified with a subspace of $X$ of dimension $\dim X-\dim X\cap Y$ (note that
we do not require $Y$ to be a subspace of $X$),
which is formed by including a single suitably chosen vector from each class of the equivalence relation $\sim$.

We assume that the reader is familiar with standard matroid terminology,
which can be found e.g. in~\cite{Oxl11}, and
we just review notation that may occasionally differ.
If $M$ is a matroid,
then $\lvert M\rvert$ denotes the number of elements of $M$.
The rank of a set $X$ of elements of $M$ is denoted by $\rank_M X$;
we drop the subscript $M$ if the matroid $M$ is clear from context.
We slightly abuse the notation and write $\rank M$ for the rank of the matroid $M$.
If $M$ is a matroid and $e$ is an element of $M$,
then $M/e$ is the matroid obtained by contracting the element $e$ and
$M\setminus e$ is the matroid obtained by deleting the element $e$.
Throughout the paper, all contractions involve non-loop elements of a matroid only.
A \emph{minor} of a matroid $M$ is any matroid that
can be obtained from $M$ contracting and deleting elements, and
the \emph{restriction} of a matroid $M$ to a set $X$ of its elements
is the matroid with the ground set $X$ and
a subset of $X$ being independent in the restriction if and only if it is independent in $M$.
If $\MM$ is a class of matroids,
then the \emph{minor closure} of $\MM$ is the class containing all matroids that are minors of an element of $\MM$.
Finally, a \emph{component} of a matroid is an inclusion-wise maximal subset of elements of $M$ such that
any two of them are contained in a common circuit.

A matroid $M$ is \emph{$\FF$-representable}
if it is a vector matroid over a field $\FF$,
i.e., its elements can be associated with vectors over $\FF$ and
a subset of elements of $M$ is independent if and only if the 
corresponding vectors are linearly independent;
such an association is referred to as an \emph{$\FF$-representation}.
If an $\FF$-representation of $M$ is fixed,
we say that the matroid $M$ is an \emph{$\FF$-represented} matroid and
identify the elements of $M$ with the vectors over $\FF$.
We will always assume that the dimension of the vector space is equal to $\rank M$,
that is, the vectors are elements of $\FF^{\rank M}$.
An $\FF$-represented matroid $M$ can be viewed as a matrix whose columns are the vectors of $M$.
Note that the matroid $M$ is preserved by row operations,
i.e., after performing any row operations
the obtained matrix is still an $\FF$-representation of the same matroid $M$.
The $\FF$-representation of the matroid $M$ obtained by deleting an element $e$
can be obtained from the $\FF$-representation of $M$ by deleting the column corresponding to $e$, and
the $\FF$-representation of the matroid $M$ obtained by contracting $e$
can be obtained from the $\FF$-representation of $M$
by first performing row operations to make the column corresponding to $e$ to be a unit vector and
then by deleting both the only row containing $1$ in the column corresponding to $e$ and the column corresponding to $e$.

We next define the connectivity function for a matroid $M$ and some functions derived from it.
If $X$ and $Y$ are disjoint subsets of the elements of a matroid $M$,
we define a two-parameter function $\lambda_M(\cdot,\cdot)$ that
\[\lambda_M(X,Y)=\rank_M X+\rank_M Y-\rank_M \paren*{X\cup Y}.\]
The \emph{connectivity function} is the one-parameter function $\lambda_M(\cdot)$ such that
$\lambda_M(X)=\lambda_M(X,Y)$ where $Y$ is the complement of $X$.
We also use $\lambda_M$ for this one-parameter function as
it is derived from the more general two-parameter function introduced earlier;
we believe that no confusion can arise as the two functions have a different number of parameters.
Finally, if $X_1,\ldots,X_k$ are disjoint subsets of the elements of a matroid $M$,
we set
\[\lambda^*_M(X_1,\ldots,X_k)=\max_{I\subseteq [k]}\lambda_M\left(\bigcup_{i\in I}X_i,\bigcup_{i\in [k]\setminus I}X_i\right).\]
In case the matroid $M$ is clear from the context we will drop the subscript,
so we will occasionally write $\lambda(\cdot)$, $\lambda(\cdot,\cdot)$ and $\lambda^*(\cdot,\ldots,\cdot)$
instead of $\lambda_M(\cdot)$, $\lambda_M(\cdot,\cdot)$ and $\lambda_M^*(\cdot,\ldots,\cdot)$.
We extend the definition of $\lambda^*(\cdot,\ldots,\cdot)$ to vector spaces and
set
\[\lambda^*(X_1,\ldots,X_k)=\max_{I\subseteq [k]} \paren*{\dim\sum_{i\in I}X_i+\dim\sum_{i\in [k]\setminus I}X_i - \dim \paren*{X_1+\ldots+X_k}}\]
for vector spaces $X_1,\ldots,X_k$.

We now define depth parameters used throughout this paper.
The \emph{contraction-deletion-depth} a matroid $M$, which is denoted by $\cdd(M)$, is defined recursively as follows:
\begin{itemize}
\item If $M$ has a single element, then $\cdd(M)=1$.
\item If $M$ is not connected, then $\cdd(M)$ is the maximum contraction-deletion-depth of a component of $M$.
\item Otherwise, $\cdd(M)=1+\min\limits_{e\in M}\min\{\cdd(M/e),\cdd(M\setminus e)\}$,
      i.e., $\cdd(M)$ is one plus the minimum contraction-deletion-depth of a matroid that
      can be obtained from $M$ by deleting or contracting a single element.
\end{itemize}      
The \emph{contraction-depth} of a matroid $M$, denoted by $\cd(M)$, is defined analogously
but with only contractions of elements permitted in the last case, and
the \emph{deletion-depth} of a matroid $M$, denoted by $\dd(M)$, is defined analogously
but with only deletions of elements permitted in the last case.

A \emph{$(d,r)$-decomposition} of a matroid $M$
is a tree $T$ such that the radius of $T$ is $d$,
the leaves of $T$ are in a one-to-one correspondence to the elements of $M$, and
the following holds for every inner vertex $v$ of $T$:
$\lambda_M^*(X_1,\ldots,X_k)\le r$
where $k$ is the number of neighbors of $v$ and
$X_1,\ldots,X_k$ are the sets of elements of $M$ assigned to the leaves of the $k$ trees of the forest $T\setminus v$.
The \emph{branch-depth} of a matroid $M$, which is denoted by $\bd(M)$,
is the minimal $k$ such that $M$ has a $(d,r)$-decomposition with $d\le k$ and $r\le k$.
It can be shown~\cite[Proposition 5.3]{DevKO20} that
if a matroid $M$ is a minor of a matroid $N$,
then $\bd(M)\le\bd(N)$.
In other words, branch-depth is a minor-monotone parameter.

For the next definition, recall that the \emph{depth} of a rooted tree $T$
is the maximum number of edges on a path from the root to a leaf of $T$.
A \emph{rooted $(d,r)$-decomposition} of a matroid $M$
is a rooted tree $T$ such that the depth of $T$ is $d$,
the leaves of $T$ are in a one-to-one correspondence with the elements of $M$, and
the following holds for every inner vertex $v$ of $T$:
$\lambda_M^*(X_1,\ldots,X_k)\le r$
where $k$ is the number of children of $v$ and
$X_1,\ldots,X_k$ are the sets of elements of $M$ assigned to the leaves of the $k$ subtrees rooted at the children of $v$.

We conclude this section with the following proposition
which relates rooted decompositions and decompositions of matroids.

\begin{proposition}
\label{prop:decomp}
If a matroid $M$ has a $(d,r)$-decomposition,
then the matroid $M$ has a rooted $(d,r)$-decomposition.
\end{proposition}

\begin{proof}
Fix a $(d,r)$-decomposition $T$ of $M$, and
let $w$ be an inner vertex of $T$ such that any other vertex of $T$ is at distance at most $d$ from $w$.
Since the radius of $T$ is equal to $d$, there is some vertex of $T$ at distance exactly
$d$ from $w$, and this vertex must be a leaf.
Let $T'$ be the rooted tree obtained from $T$ by rooting it at $w$.
Note that the depth of $T'$ is $d$.
Consider now an inner vertex $v$ of $T'$ and let $k$ be the number of children of $v$.
Further, let $X_1,\ldots,X_k$ be the sets of elements of $M$
assigned to the leaves of the $k$ subtrees rooted at the children of $v$, and
let $Y$ be the set of elements of $M$ not assigned to a leaf of the subtree rooted at $v$,
namely, the elements mapped to a leaf not in any of the $k$ subtrees.
Since $T$ is a $(d,r)$-decomposition and
it holds that
\[\lambda_M\left(\bigcup_{i\in I}X_i,\bigcup_{i\in [k]\setminus I}X_i\right)\le
  \lambda_M\left(\bigcup_{i\in I}X_i,Y\cup\bigcup_{i\in [k]\setminus I}X_i\right)\]
for any subset $I\subseteq [k]$,
where the latter is at most $r$,
we obtain that $\lambda^*_M(X_1,\ldots,X_k)\le r$.
It follows that $T'$ is a rooted $(d,r)$-decomposition of the matroid $M$.
\end{proof}

\section{Main lemma}
\label{sec:lemma}

The core of our argument is Lemma~\ref{lm:main} which we prove in this section.
Informally speaking, the lemma asserts that
there are only two reasons
why a partition $E_1,\ldots,E_k$ of elements of a represented matroid $M$
satisfies that $\lambda^*_M(E_1,\ldots,E_k)\le r$:
\begin{itemize}
\item the subspaces $X_1,\ldots,X_k$ spanned by $E_1,\ldots,E_k$ intersect in a small dimensional subspace, or
\item the subspaces $X_1,\ldots,X_k$ form an ``almost direct'' sum of vector spaces,
      i.e., they are subspaces in a general position in a space whose dimension is slightly smaller than
      the sum of the dimensions of the subspaces $X_1,\ldots,X_k$.
\end{itemize}
It is not hard to see that either of these two properties implies that $\lambda^*_M(E_1,\ldots,E_k)$ is small.
Indeed, if $X_1,\ldots,X_k$ are subspaces of a vector space and
there exists a vector space $A$ of dimension $a$ such that
\begin{equation}
\sum_{i=1}^k\dim \paren*{X_i/A} \le\dim \paren*{\paren*{\sum_{i=1}^k X_i}/A}+b,\label{eq:AB}
\end{equation}
then $\lambda^*(X_1,\ldots,X_k)\le a+b$.
Note that the parameter $a$ represents the dimension of the subspace ``shared'' by $X_1,\ldots,X_k$,
which is the subspace $A$, and
the parameter $b$ represents how close the subspaces are to the direct sum (after contracting the shared subspace).

We next state the main lemma of this section. 
Note that, assuming that $A\subseteq X_1+\cdots+X_k$,
it holds that
\[\dim\paren*{\sum\limits_{i=1}^k X_i} =\dim\paren*{\paren*{\sum\limits_{i=1}^k X_i}/A}+\dim A.\]
Hence, the lemma says that \eqref{eq:AB} holds with $a=b=\dim A\le 3r$ for the subspace $A$ from the statement of the lemma.

\begin{lemma}
\label{lm:main}
Let $X_1,\ldots,X_k$ be vector spaces over a field $\FF$ such that $\lambda^*(X_1,\ldots,X_k)\le r$.
There exists a vector space $A$ with dimension at most $3r$ such that
\[\sum_{i=1}^k\dim \paren*{X_i/A} \le \dim \paren*{\sum_{i=1}^k X_i}.\]
\end{lemma}

\begin{proof}
We will construct vector spaces $A_m$, $L_m$ and $R_m$ for every $m\in [k]$
with the following properties:
\begin{enumerate}[(i)]
\item\label{it:1} $A_m$ is a subspace of $X_1+\cdots+X_m$,
\item\label{it:2} there exists $I\subseteq [m]$ such that
      \[L_m=\sum_{i\in I}X_i \quad \mbox{and} \quad R_m=\sum_{i\in [m]\setminus I}X_i,\]
\item\label{it:3} if $m>1$, then $A_{m-1}\subseteq A_m$, $L_{m-1}\subseteq L_m$, and $R_{m-1}\subseteq R_m$,
\item\label{it:4} $L_m\cap R_m$ is a subspace of $A_m$,
\item\label{it:5} $\dim A_m\le 3\cdot\dim L_m\cap R_m$, and
\item\label{it:6} it holds that
  \[\sum_{i=1}^m\dim \paren*{X_i/A_m}\le \dim \paren*{\sum_{i=1}^m X_i}.\]
\end{enumerate}
Once we have constructed $A_m$, $L_m$ and $R_m$ for every $m \in [k]$,
we set the vector space $A$ from the statement of the lemma to be the vector space $A_k$.
Note that
since the dimension of $L_k\cap R_k$ cannot be larger than $\lambda^*(X_1,\ldots,X_k)\le r$ by the property (\ref{it:2}),
the dimension of $A=A_k$ is at most $3r$ by the property (\ref{it:5}), and
the inequality from the statement of the lemma follows from the property (\ref{it:6}) above for $m=k$.
Therefore, it suffices to construct the vector spaces $A_m$, $L_m$, and $R_m$, $m\in [k]$, in order to prove the lemma.

We now iteratively construct the vector spaces $A_m$, $L_m$, and $R_m$, starting with $m=1$.
We set $A_1$ and $R_1$ to be the $0$-dimensional vector space and $L_1=X_1$.
It is straightforward to verify that $A_1$, $L_1$ and $R_1$ satisfy all properties (\ref{it:1})--(\ref{it:6}).

Suppose that we have constructed vector spaces $A_{m-1}$, $L_{m-1}$, and $R_{m-1}$.
Consider the vector space $A'_m$ defined as follows:
\begin{equation}
A'_m=\left(X_m+A_{m-1}\right)\cap\sum_{i \in [m-1]}X_i.\label{eq:Am}
\end{equation}
Note that $A'_m$ contains $A_{m-1}$ as a subspace, since $A_{m-1}$ is a subspace
of $\sum_{i = 1}^{m - 1}X_i$.
We next distinguish two cases:
the first is that the spaces $A_{m-1}$ and $A'_m$ are equal, and
the second is that $A_{m-1}$ is a proper subspace of $A'_m$.

If the spaces $A_{m-1}$ and $A'_m$ are equal,
we set the subspaces $A_m$, $L_m$ and $R_m$ as follows:
\[L_m=L_{m-1}+X_m, \quad R_m=R_{m-1} \quad \mbox{and} \quad A_m=A'_m+(L_m\cap R_m)=A_{m-1}+(L_m\cap R_m).\]
The subspaces $A_m$, $L_m$ and $R_m$ clearly satisfy the properties (\ref{it:1})--(\ref{it:4}).
The property (\ref{it:5}), which bounds the dimension of $A_m$, follows from the next estimates:
\begin{align*}
  \dim A_m & = \dim \paren*{A_{m-1}+(L_m\cap R_m)} \\
           & = \dim A_{m-1} + \dim \paren*{L_m\cap R_m} - \dim \paren*{A_{m-1}\cap (L_m\cap R_m)} \\
           & \le \dim A_{m-1} + \dim \paren*{L_m\cap R_m} - \dim \paren*{A_{m-1}\cap (L_{m-1}\cap R_{m-1})} \\
           & = \dim A_{m-1} + \dim \paren*{L_m\cap R_m} - \dim \paren*{L_{m-1}\cap R_{m-1}} \\
           & \le 3\dim \paren*{L_{m-1}\cap R_{m-1}} + \dim \paren*{L_m\cap R_m} - \dim \paren*{L_{m-1}\cap R_{m-1}} \\
           & = \dim \paren*{L_m\cap R_m} + 2 \dim \paren*{L_{m-1}\cap R_{m-1}} \le 3 \dim\paren*{ L_m\cap R_m}.\\
\end{align*}
We prove the property (\ref{it:6}) jointly for the first and the second cases later.

If the space $A_{m-1}$ is a proper subspace of $A'_m$, we proceed as follows.
Let $d=\dim A'_m-\dim A_{m-1}$;
we remark that the previous case $A'_m=A_{m-1}$
can be viewed as the case $d=0$ of the arguments that we next present, however,
we decided to separate the two cases in the interest of the clarity of presentation.
The definition of $A'_m$ as \eqref{eq:Am} implies that
there exist (linearly independent) vectors $x_1,\ldots,x_d\in X_m$ such that
any basis of $A_{m-1}$ together with $x_1,\ldots,x_d$ is a basis of $A'_m$.
Note that none of the vectors $x_1,\ldots,x_d$ belongs to $A_{m-1}$.
Since the vector $x_i$ belongs to $A'_m$, and
so $x_i\in X_1+\cdots+X_{m-1}=L_{m-1}+R_{m-1}$,
there exist $\ell_i\in L_{m-1}$ and $r_i\in R_{m-1}$ such that $x_i=\ell_i+r_i$.
Let $L'$ be the linear hull of the vectors $\ell_1,\ldots,\ell_d$ and
$R'$ the linear hull of the vectors $r_1,\ldots,r_d$.
Note that the vectors $x_1,\ldots,x_d$ are contained in $L'+R'$,
which implies that
\begin{equation}
  \dim (L'+R')-\dim \paren*{(L'+R')\cap A_{m-1}}\ge d.\label{eq:lm1}
\end{equation}
On the other hand,
the space $L'\cap R'$ is a subspace of $L_{m-1}\cap R_{m-1}$ and so of $A_{m-1}$,
which yields that
\begin{equation}
  \dim (L'\cap R')-\dim \paren*{(L'\cap R')\cap A_{m-1}}=0.\label{eq:lm2}
\end{equation}
Next observe that
\begin{align}
\dim L'+\dim R' & = \dim (L'+R')+\dim (L'\cap R'), \mbox{ and} \label{eq:lm3}\\
\dim \paren*{L'\cap A_{m-1}} + \dim \paren*{R'\cap A_{m-1}} & = \dim \paren*{(L'+R')\cap A_{m-1}}+\dim \paren*{(L'\cap R')\cap A_{m-1}}. \label{eq:lm4}
\end{align}
We combine \eqref{eq:lm1}, \eqref{eq:lm2}, \eqref{eq:lm3} and \eqref{eq:lm4} to obtain that
\[\left(\dim L'-\dim \paren*{L'\cap A_{m-1}}\right)+\left(\dim R'-\dim \paren*{R'\cap A_{m-1}}\right)\ge d.\]
It follows that $\dim L'-\dim \paren*{L'\cap A_{m-1}}\ge d/2$ or $\dim R'-\dim \paren*{R'\cap A_{m-1}}\ge d/2$ (or possibly both).
The two possibilities are symmetric, so we assume that $\dim L' - \dim \paren*{L' \cap A_{m - 1}} \geq d/2$.

We define the subspaces $L_m$, $R_m$, and $A_m$ as
\[L_m=L_{m-1}, \quad R_m=R_{m-1}+X_m \quad \mbox{and} \quad A_m=A'_m+(L_m\cap R_m).\]
Note that every $\ell_i=x_i-r_i$ belongs to $L_m\cap R_m$:
indeed, $\ell_i\in L_{m-1}\subseteq L_m$ and $x_i-r_i\in X_m + R_{m - 1} = R_m$.
Hence, the space $L'$, which is the linear hull of the vectors $\ell_1,\ldots,\ell_d$,
is a subspace of $L_m\cap R_m$.
Since $L_{m-1}\cap R_{m-1}$ is a subspace of $A_{m-1}$,
it follows that
\begin{align*}
  \dim \paren*{L_m\cap R_m} - \dim\paren*{ L_{m-1}\cap R_{m-1}}
  & \ge \dim \paren*{L'+(L_{m-1}\cap R_{m-1})} - \dim \paren*{L_{m-1}\cap R_{m-1}}\\
  & \ge \dim \paren*{L'+A_{m-1}}-\dim A_{m-1}\\
  & =\dim L'-\dim \paren*{L'\cap A_{m-1}}\\
  & \ge d/2.
\end{align*}  
Hence, 
it holds that $d/2\le\dim \paren*{L_m\cap R_m} - \dim\paren*{ L_{m-1}\cap R_{m-1}}$,
which is equivalent to 
\begin{equation}
d\le 2\left(\dim \paren*{L_m\cap R_m} - \dim\paren*{ L_{m-1}\cap R_{m-1}}\right).\label{eq:upperd}
\end{equation}.
We can now estimate the dimension of $A_m$ as follows (recall that $L_{m-1}\cap R_{m-1}$
is a subspace of $A_{m-1}\subseteq A'_m$ and also of $L_m\cap R_m$):
\begin{align*}
  \dim A_m & \le \dim A'_m+\dim \paren*{L_m\cap R_m} - \dim \paren*{L_{m-1}\cap R_{m-1}}\\
           & = \dim A_{m-1}+d+\dim \paren*{L_m\cap R_m} - \dim\paren*{L_{m-1}\cap R_{m-1}}\\
           & \le \dim A_{m-1}+3\left(\dim \paren*{L_m\cap R_m} - \dim\paren*{L_{m-1}\cap R_{m-1}}\right)\\
           & \le 3\dim\paren*{L_{m-1}\cap R_{m-1}}+3\left(\dim\paren*{L_m\cap R_m} - \dim\paren*{L_{m-1}\cap R_{m-1}}\right)\\
           & = 3\dim \paren*{L_m\cap R_m},
\end{align*}	 
where the second inequality follows from \eqref{eq:upperd}.
We have now verified that the properties (\ref{it:1})--(\ref{it:5}) hold in the second case.

We now continue jointly the analysis of the first and second cases and establish the property (\ref{it:6}).
We will use the definition of $A'_m$ as \eqref{eq:Am} and that $A_m=A'_m+(L_m\cap R_m)$.
The definition of $A'_m$ yields that
\begin{equation}
  \dim\paren*{X_m+A_{m-1}} + \dim\paren*{\sum_{i=1}^{m-1}X_i} = \dim\paren*{\sum_{i=1}^m X_i} + \dim A'_m;\label{eq:lmm}
\end{equation}
indeed,
\eqref{eq:lmm} follows from the equality $\dim S+\dim T=\dim S\cup T+\dim S\cap T$
used for $S=X_m+A_{m-1}$ and $T=\sum\limits_{i=1}^{m-1}X_i$.
We now establish the property (\ref{it:6}).
\begin{align*}
  \sum_{i=1}^m\dim\paren*{X_i/A_m} & \le \sum_{i=1}^m\dim\paren*{X_i/A'_m}\\
        & = \sum_{i=1}^{m-1}\dim\paren*{X_i/A'_m}+\dim \paren*{X_m/A_{m-1}}-\left(\dim A'_m-\dim A_{m-1}\right)\\
			 & = \sum_{i=1}^{m-1}\dim (X_i/A'_m) + \dim (X_m+A_{m-1})-\dim A_{m-1}-\dim A'_m + \dim A_{m-1}\\
			 & = \sum_{i=1}^{m-1}\dim (X_i/A'_m) + \dim (X_m+A_{m-1})-\dim A'_m\\
			 & \le \sum_{i=1}^{m-1}\dim(X_i/A_{m-1}) + \dim (X_m+A_{m-1})-\dim A'_m\\
			 & \le \sum_{i=1}^{m-1}\dim X_i+\dim (X_m+A_{m-1})-\dim A'_m\\
       &= \sum_{i = 1}^{m} \dim X_i,
\end{align*}			 
where the last equality is a restatement of \eqref{eq:lmm}.
\end{proof}

\section{Main result}
\label{sec:main}

In this section, we prove the main result of our paper
which is the next theorem.
The results announced in Section~\ref{sec:intro} are its immediate corollaries
which we derive just after the proof of the theorem.

\begin{theorem}
\label{thm:main}
If an $\FF$-representable matroid $M$ has a rooted $(d,r)$-decomposition,
then there exists an $\FF$-representable matroid $N$ such that
$M$ is a minor of $N$ and $\cdd(N)\le 2r\cdot (4^d-1)+1$.
\end{theorem}

\begin{proof}
The proof proceeds by induction on $d$.
We will show the statement for $\FF$-represented matroids and
we additionally show that
the considered $\FF$-representation of $M$
can be obtained from an $\FF$-representation of $N$ (by deleting rows corresponding to contracted elements and
deleting columns corresponding to deleted elements of $N$).
The base of the induction is the case when $d=0$.
If $d=0$,
then the matroid $M$ has a single element and we set $N=M$ (note that $\cdd(N)=1$ as required).
We next present the induction step.
Fix an $\FF$-represented matroid $M$ and a rooted $(d,r)$-decomposition $T$ of $M$.
Let $k$ be the number of children of the root of $T$ and
let $T_1,\ldots,T_k$ be the (rooted) subtrees of $T$ rooted at the children of $T$.
Further, let $E_i$, be the elements of $M$ assigned to the leaves of $T_i$, $i\in [k]$, and
let $X_i$ be the linear hull of $E_i$.

Since $T$ is a $(d,r)$-decomposition of $M$,
it holds that $\lambda^*(X_1,\ldots,X_k)\le r$.
Hence, Lemma~\ref{lm:main} yields that
there exists a subspace $A$ of $\FF^{\,\rank M}$ with dimension at most $3r$ such that
\begin{equation}
\sum_{i=1}^k\dim (X_i/A) \le \dim \sum_{i=1}^k X_i.\label{eq:sumdim}
\end{equation}
Let $d_A$ be the dimension of the vector space $A$ and let $d_i$ be the dimension of the vector space $X_i/A$, $i\in [k]$.
Note that $\rank M$ is equal to $\dim (X_1+\cdots+X_k)$, which is at most $d_1+\cdots+d_k+d_A$.
On the other hand, the rank of $M$ is at least $d_1+\cdots+d_k$ by \eqref{eq:sumdim}.

Fix a basis $B_A = \set{a_1,\ldots,a_{d_A}}$ of the vector space $A$, and
for each $i\in [k]$ let $B_{i} = \set{x^i_1,\ldots,x^i_{d_i}}$ be vectors of $X_i$, such that
$a_1,\ldots,a_{d_A},x^i_1,\ldots,x^i_{d_i}$ is a basis of the vector space $X_i+A$.
Observe that the vector space $\FF^{\,\rank M}=X_1+\cdots+X_k$
is spanned by the set of vectors $B_A \cup B_1 \cup \dots \cup B_k$.
It follows that there is a subset of these vectors
is a basis of the vector space $X_1+\ldots+X_k$;
we may assume that this subset contains all vectors $a_1,\ldots,a_{d_A}$ (as these vectors are linearly independent).
Hence, after possibly permuting the vectors $x^i_1,\ldots,x^i_{d_i}$,
we assume that for each $i \in [k]$ there exists $d'_i \in \{0\}\cup [d_i]$, such that
the set of vectors $B = B_A \cup \set*{x^i_j \ \colon \ i \in [k], j \in [d'_i]}$
form a basis of the vector space $X_1+\ldots+X_k$.
In addition, we assume that
the vectors $a_1,\ldots,a_{d_A}$ are the first $d_A$ unit vectors,
$x^1_1,\ldots,x^1_{d'_1}$ are the next $d'_1$ unit vectors, etc.
Since the sum $d'_1+\cdots+d'_k+d_A$ is equal to the rank of $M$,
which is at least $d_1+\cdots+d_k$, we obtain that
\begin{equation}
\sum_{i=1}^k (d_i-d'_i) \;\le\; d_A.\label{eq:sumdd}
\end{equation}
Let $d_C$ be the value of the sum from the left side of \eqref{eq:sumdd};
note that $d_C\le d_A\le 3r$.
We now view $X_1+\cdots+X_k$ as a subspace of $\FF^{\,\rank M+d_C}$,
where the vectors in $B$
are the first $\rank M$ unit vectors (in the order described above).
Next, let $C = \set*{b^i_j \ \colon \ i \in [k], j \in [d_i] \setminus [d'_i]}$ be the remaining (last) $d_C$ 
unit vectors of $\FF^{\,\rank M + d_C}$ (where $b^1_j$ comes before $b^2_j$, which comes before $b^3_j$, etc).

\begin{figure}
\begin{center}
\epsfbox{cddrepr-1.mps}
\end{center}
\caption{The $\FF$-representation of the matroid $M'$ constructed in the proof of Theorem~\ref{thm:main}
         when $k=3$, $d_A=4$, $d'_1=3$, $d_1=4$, $d'_2=2$, $d_2=4$, $d'_3=2$ and $d_3=3$.
	 Stars depict entries that can be arbitrary (both zero or non-zero), and
	 the entries that are not displayed are zero.
	 The representation of the matroid $M''$ from the proof
	 is the part of the representation of $M'$ encompassed by the dashed lines.}
\label{fig:M}
\end{figure}

We next construct a vector matroid $M'$ of rank equal to $\rank M+d_C$ such that $M$ is a minor of $M'$.
The matroid $M'$ contains the following $\lvert M\rvert+d_A+d_C$ vectors:
\begin{itemize}
\item For every $i\in [k]$ and every element $x\in E_i$, the matroid $M'$ contains the vector
      \[\sum_{j=1}^{d_A}\alpha_j a_j+\sum_{j=1}^{d'_i}\beta_j x^i_j+\sum_{j=d'_i+1}^{d_i} \beta_j b^i_j\]
      where $\alpha_j\in\FF$, $j\in [d_A]$, and $\beta_j\in\FF$, $j\in [d_i]$
      are the unique coefficients such that
      \[x=\sum_{j=1}^{d_A}\alpha_j a_j+\sum_{j=1}^{d_i}\beta_j x^i_j.\]
      In particular,
      the matroid $M'$ contains the unit vectors $x^i_j$, $i\in [k]$ and $j\in [d'_i]$ and
      the unit vectors $b^i_j$, $i\in [k]$ and $j\in [d_i]\setminus [d'_i]$;
      the unit vectors $b^i_j$, $i\in [k]$ and $j\in [d_i]\setminus [d'_i]$
      are included because of the elements $x^i_j$ with $j\in [d_i]\setminus [d'_i]$.
\item The $d_A$ vectors $a_1,\ldots,a_{d_A}$.
\item For every $i\in [k]$ and $j\in [d_i]\setminus [d'_i]$,
      the matroid $M'$ contains the vector $z^i_j=b^i_j-x^i_j$;
      note that there are exactly $d_B$ vectors added to $M'$ because of this rule.
\end{itemize}
The construction of the matroid $M'$ is illustrated in Figure~\ref{fig:M}.

Observe that the matroid $M$ is a minor of $M'$,
specifically,
the matroid $M$ can be obtained from $M'$ by deleting the $d_A$ elements $a_1,\ldots,a_{d_A}$ and
contracting the $d_B$ elements $z^i_j$, $i\in [k]$ and $j\in [d_i]\setminus [d'_i]$.
Let $M''$ be the matroid obtained from $M'$
by contracting the $d_A$ elements $a_1,\ldots,a_{d_A}$ and
deleting the $d_B$ elements $z^i_j$, $i\in [k]$ and $j\in [d_i]\setminus [d'_i]$.
Let $M''_i$ be the restriction of $M''$ to $E_i$, and
observe $M''_i$ is a union of components of $M''$.

We aim to apply induction to each of the matroids $M''_i$, $i\in [k]$.
To do so, we first show that
$T_i$ is a rooted $(d-1,4r)$-decomposition of the matroid $M''_i$.
Observe that the matroid $M''_i$ can be obtained from the matroid $M_i$
by first adding the $d_A$ elements $a_1,\ldots,a_{d_A}$ and
then contracting the $d_A$ elements $a_1,\ldots,a_{d_A}$.
It follows that
\[\rank_M X-d_A\le \rank_M X-\dim A\cap\overline{X}=\rank_{M''_i} X\le \rank_M X\]
for every $X\subseteq E_i$.
Hence, the following holds for any two disjoint subsets $X$ and $Y$ of $E_i$:
\begin{align*}
\lambda_{M''_i}(X,Y) & = \rank_{M''_i} X+\rank_{M''_i} Y-\rank_{M''_i} (X\cup Y)\\
                     & \le \rank_{M''_i} X+\rank_{M''_i} Y-\left(\rank_{M} (X\cup Y)-d_A\right)\\
		     & \le \rank_M X+\rank_M Y-\rank_{M} (X\cup Y) + d_A\\
		     & = \lambda_M(X,Y)+\dim A \le \lambda_M(X,Y)+3r.
\end{align*}
We conclude that $T_i$ is a rooted $(d-1,4r)$-decomposition of the matroid $M''_i$.

By induction, the $\FF$-represented matroid $M''_i$ is a minor of an $\FF$-represented matroid $N_i$ such that
\[\cdd(N_i)\le 8r\cdot (4^{d-1}-1)+1=2r\cdot(4^d-1)+1-6r.\]
Set $n_i=\rank N_i-\rank M''_i$, $i\in [k]$, and
let $y^i_j$, $i\in [k]$ and $j\in [n_i]$, be the elements of $N_i$ that were contracted.
Note that for every $i\in [k]$
the elements $x^i_j$, $j\in [d_i]$, and $y^i_j$, $j\in [n_i]$, form a basis of $N_i$.

\begin{figure}
\begin{center}
\epsfbox{cddrepr-2.mps}
\end{center}
\caption{The $\FF$-representation of the matroid $N$ constructed in the proof of Theorem~\ref{thm:main}
         when $k=3$, $d_A=4$, $d'_1=3$, $d_1=4$, $n_1=1$, $d'_2=2$, $d_2=4$, $n_2=1$, $d'_3=2$, $d_3=3$ and $n_3=2$.
	 Stars depict entries that can be arbitrary (both zero or non-zero), and
	 the entries that are not displayed are zero.
	 The representations of matroids $N_1$, $N_2$ and $N_3$ obtained from induction
	 are encompassed by the dashed lines.}
\label{fig:N}
\end{figure}

We will now construct the sought matroid $N$;
the construction of the matroid $N$ is illustrated in Figure~\ref{fig:N}.
The matroid $N$ is a vector matroid of rank $R=d_1+\cdots+d_k+n_1+\cdots+n_k+d_A$
with $\lvert N_1\rvert+\cdots+\lvert N_K\rvert+d_A+d_B$ elements,
which one-to-one correspond to the elements of the matroids $N_1,\ldots,N_k$, and
the elements $a_1,\ldots,a_{d_A}$ and $z^i_j$, $i\in [k]$ and $j\in [d_i]\setminus [d'_i]$, and
so we identify them with these elements.
The elements $a_1,\ldots,a_{d_A}$,
the elements $x^i_j$, $i\in [k]$ and $j\in [d_i]$, and
the elements $y^i_j$, $i\in [k]$ and $j\in [n_i]$ will form a basis of $N$, and
so we can assume that these elements
are represented by unit vectors in the vector space $\FF^{d_1+\cdots+d_k+n_1+\cdots+n_k+d_A}$ and
the elements $a_1,\ldots,a_{d_A}$ are the first $d_A$ unit vectors.
The vector matroid $N$ contains the following vectors:
\begin{itemize}
\item each of the unit vectors $a_1,\ldots,a_{d_A}$,
\item the vector $z^i_j$, $i\in [k]$ and $j\in [d_i]\setminus [d'_i]$
      with all its additional $n_1+\cdots+n_k$ entries set to zero,
\item for every $i\in [k]$ and every vector of $N_i$ contained in $M''_i$,
      this vector with its first $d_A$ entries being the same as in $M'$ and
      all its additional $R-d_A-d_i-n_i$ entries set to zero, and
\item for every $i\in [k]$ and every vector of $N_i$ not contained in $M''_i$,
      this vector with all its the additional $R-d_i-n_i$ entries set to zero.
\end{itemize}
Note that the matroid obtained from the matroid $N$
by contracting the $d_A$ elements $a_1,\ldots,a_{d_A}$ and
deleting the $d_B$ elements $z^i_j$, $i\in [k]$ and $j\in [d_i]\setminus [d'_i]$
is the matroid that is the union of the matroids $N_1,\ldots,N_k$.
Since $d_A+d_B\le 6r$ and $\cdd(N_i)\le 2r\cdot(4^d-1)+1-6r$ for every $i\in [k]$,
it follows that the contraction-deletion-depth of $N$ is at most $2r\cdot(4^d-1)+1$.
On the other hand,
the matroid obtained from the matroid $M$
by contracting the $d_B$ elements $z^i_j$, $i\in [k]$ and $j\in [d_i]\setminus [d'_i]$,
contracting the $n_1+\cdots+n_k$ elements $y^i_j$, $i\in [k]$ and $j\in [n_i]$ and
deleting all the remaining elements not contained in $E_1\cup\cdots\cup E_k$
yields the matroid $M$.
This completes the proof of the induction step and so of the theorem.
\end{proof}

We immediately obtain Theorem~\ref{thm:bound} as a corollary.

\begin{proof}[Proof of Theorem~\ref{thm:bound}]
Let $M$ be an $\FF$-representable matroid with branch-depth $k$.
Since $M$ has a $(d,r)$-decomposition for some non-negative integers $d\le k$ and $r\le k$,
Proposition~\ref{prop:decomp} yields that $M$ has a rooted $(d,r)$-decomposition.
The corollary now directly follows from Theorem~\ref{thm:main}.
\end{proof}

We also obtain the characterization of classes of representable matroids with bounded branch-depth given in Theorem~\ref{thm:class}.

\begin{proof}[Proof of Theorem~\ref{thm:class}]
Fix a class $\MM$ of $\FF$-representable matroids.
Suppose that there exists a class $\MN$ of matroids with contraction-deletion-depth bounded by $d$ such that
$\MM$ is a subclass of the minor closure of the class $\MN$.
Since the contraction-deletion-depth is an upper bound on the branch-depth~\cite[Theorem 5.2]{DevKO20},
the branch-depth of every matroid contained in $\MN$ is also bounded by $d$ and
so is the branch-depth of every matroid contained in the minor closure of $\MN$ as
the branch-depth is minor-monotone~\cite[Proposition 5.3]{DevKO20}.
Hence, if $\MM$ is a subclass of the minor closure of $\MN$,
the branch-depth of every matroid contained in $\MM$ is bounded by $d$.

We now prove the other implication.
Suppose that the branch-depth of every matroid contained in $\MM$ is bounded by $d'$.
For every matroid $M\in\MM$,
let $N_M$ be an $\FF$-representable matroid with contraction-deletion-depth bounded at most $2d'\cdot (4^{d'}-1)+1$ such that
$M$ is a minor of $N_M$;
such a matroid $N_M$ exists by Theorem~\ref{thm:bound}.
The class $\MN$ is formed by the matroids $N_M$ for all $M\in\MM$.
Since the contraction-deletion-depth of every matroid contained in $\MN$ is at most $2d'\cdot (4^{d'}-1)+1$ and
the class $M$ is a subclass of the minor closure of $\MN$,
the other implication is now proven.
\end{proof}

\section{Conclusion}
\label{sec:concl}

Since our main result, which asserts that every representable matroid with bounded branch-depth
is a minor of a (representable) matroid with bounded contraction-deletion-depth,
applies to representable matroids
it is natural to ask whether it can be extended to all matroids.
We believe that indeed the same statement, perhaps with a worse functional dependance on branch-depth,
also holds for general matroids.

\begin{conjecture}
\label{conj1}
There exists a function $f$ such that
if a matroid $M$ has branch-depth $d$,
then there exists a matroid $N$ with contraction-deletion-depth at most $f(d)$ such that
$M$ is a minor of $N$.
\end{conjecture}

It is possible to state \Cref{conj1} purely in terms of relating classes of matroids of
bounded branch-depth and the minor closures of matroid classes with bounded contraction-deletion-depth.
It is easy to see that \Cref{conj1} and \Cref{conj2} are equivalent.

\begin{conjecture}[\Cref{conj1}, restated]
\label{conj2}
A class $\MM$ of matroids has bounded branch-depth if and only if
there exists a class $\MN$ of matroids with bounded contraction-deletion-depth such that
$\MM$ is a subclass of the minor closure of $\MN$.
\end{conjecture}

We have not been able to construct representable matroids $M_i$ that
have $(d,r)$-decompo\-si\-tions and
there would not exist matroids $N_i$ with contraction-deletion-depth $\OO(dr)$ containing $M_i$ as a minor.
So, we pose the following as an open problem.

\begin{problem}
\label{prob}
Is every representable matroid $M$ with branch-depth at most $d$ a minor of
a (representable) matroid $N$ with contraction-deletion-depth at most $Kd^2+L$
for some $K$ and $L$ independent of $M$?
\end{problem}

\bibliographystyle{bibstyle}
\bibliography{cddrepr}

\end{document}